\documentclass[10pt,leqno,oneside]{amsart}
\usepackage{xcolor} 
\usepackage{amsthm,amsmath,amssymb,latexsym,bussproofs,tikz,xcolor,enumitem,pifont,scrextend}
\usepackage{amsthm}
\usepackage{mathrsfs}
\usepackage{enumitem}
\usepackage{stmaryrd}
\usepackage{bussproofs}
\usepackage{phonetic}
\usepackage{upgreek}
\usepackage[applemac]{inputenc}
\usepackage{cite}  

  \newcommand{\basic}{{\sf Basic}}
  \newcommand{\tsn}{{\sf TS}_0}
  \newcommand{\utso}{{\sf UTS}_0}
  \newcommand{\pkf}{{\sf PKF}}
  \newcommand{\kf}{{\sf KF}}
  \newcommand{\bdm}{\sf BDM}

\newcommand{\T}{{\sf T}}

\newcommand{\lra}{\leftrightarrow}

\long\def\symbolfootnote[#1]#2{\begingroup%
\def\thefootnote{\fnsymbol{footnote}}\footnote[#1]{#2}\endgroup}

\newcommand{\ra}{\rightarrow}

\newcommand{\Ra}{\Rightarrow}

\newcommand{\mbb}{\mathbb}

\newcommand{\vphi}{\varphi}
\newcommand{\sth}{\;|\;}

\newcommand{\urc}{\urcorner}
\newcommand{\ulc}{\ulcorner}

\newcommand{\mc}{\mathcal}

\newcommand{\bigWedge}{\mbox{\smaller$\,\bigwedge\,$}}\newcommand{\bigVee}{\mbox{\smaller$\,\bigvee\,$}}

\newcommand{\sent}{{\sf Sent}_{\mc L}}
\newcommand{\pr}{{\sf Pr}}
\newcommand{\prt}{\pr^2}
\newcommand{\lt}{\mc L_{\T}}
\newcommand{\sentt}{{\sf Sent}_{\lt}}
\newcommand*\subdot[1]{\oalign{$#1$\cr\hfil.\hfil}}%
\providecommand\dvee%
		{\mathbin{\subdot \vee}}
\providecommand\dneg%
		{\mathbin{\subdot \neg}}
\providecommand\dwedge%
		{\mathbin{\subdot \wedge}}
\providecommand\drightarrow%
		{\mathbin{\subdot \rightarrow}}
\providecommand\dnot%
		{\mathord{\subdot{\lnot}}}
\providecommand\dbot%
		{\mathord{\subdot{\bot}}}
\providecommand\dforall%
		{\mathord{\subdot\forall\!\!}}
\providecommand\dexists%
		{\mathord{\subdot\exists}}

\newcommand{\tso}{{\sf TS}_0}

\newcommand{\rr}[1]{{\sf r}({#1})}
\newcommand{\fr}[1]{{\sf R}({#1})}
\newcommand{\rtwo}[1]{{\sf R}^2(#1)}
\theoremstyle{plain}

\newtheorem{dfn}{Definition}

\newtheorem{lemma}{Lemma}
\newtheorem{corollary}{Corollary}

\newtheorem{observation}{Observation}
\newtheorem{proposition}{Proposition}

\begin{document}	

\title{Iterated reflection over full disquotational truth}
\author{Martin Fischer \& Carlo Nicolai \& Leon Horsten }
\maketitle

\begin{abstract}

\noindent Iterated reflection principles have been employed extensively to unfold epistemic commitments that are incurred by accepting a mathematical theory. Recently this has been applied to theories of truth. The idea is to start with a collection of Tarski-biconditionals and arrive by finitely iterated reflection at strong compositional truth theories. In the context of classical logic it is incoherent to adopt an initial truth theory in which $A$ and `$A$ is true' are inter-derivable.  In this article we show how in the context of a weaker logic, which we call Basic De Morgan Logic, we can coherently start with such a fully disquotational truth theory and arrive at a strong compositional truth theory by applying a natural uniform reflection principle a finite number of times.
\end{abstract}

\section{Introduction}

 In the paper we pursue the strategy of iterating reflection principles on a basic truth theory $\tso$ that encapsulates, or so we argue, the fundamental building blocks of truth-theoretic reasoning. Its components are: a theory of the objects of truth (syntax theory in our case), basic truth-theoretic principles that enable us to infer $\T\ulc \phi\urc$ from a sentence $\phi$   and vice-versa. In order to remain faithful to these assertability conditions for truth ascriptions, classical logic cannot be used on account of the liar paradox. In particular,  
one should be prepared not to have the rule of conditionalization for all sentences. Also, it is desirable to work at the right level of generality by, for instance, reasoning without committing oneself to paracomplete or paraconsistent options. We employ a logic that does this and call it, following \cite{fie08}, \emph{Basic De Morgan logic} (cf. \S\ref{cortru} for the definition). 

The theory $\tso$, however, may not be all there is to truth. Many authors have discussed further desiderata for truth -- such as full compositionality -- that are out of reach for our basic theory $\tso$. Recently, Leon Horsten and Graham Leigh  have studied  iterations of reflection over a basic theory in classical logic \cite{hol15}. Their classical starting point, however, leads to a loss of the intimate connection between truths and their assertability that is present in $\tso$. In the following we therefore extend Horsten and Leigh's strategy in the framework of Basic De Morgan logic. 

Reflection is rooted in the fundamental intuition that we are committed to the truth  of the  sentences that are provable in a theory that we accept. This operation can be expressed in different ways: as {\it global reflection}, where we make explicit use of the truth predicate, or as {\it uniform reflection}, where we express this intuition schematically without direct reference to truth. In the classical case, global and uniform reflection are provably different operations:  a variety of well-known theories  of truth can be closed under uniform reflection but not under global reflection.  However, this is not so in the case of Basic De Morgan logics and variations thereof, where global and uniform reflection coincide for a wide class of theories containing the truth principles of $\tso$. 



The theories that we study in this work, resulting from the iteration of reflection over $\tso$, can be characterized as {\it internal} axiomatizations of Kripke's fixed point construction because their principles and rules are all sound with respect to this class of models. The paradigmatic case of an internal theory is $\pkf$ \cite{hah06}. These theories are faithful to the fixed-point models and interact well with the process of reflection. There is an alternative array of theories capturing Kripke's construction  {\it externally}. They are couched in classical logic and they are meant to be faithful to the set-theoretic definition of this class of models. Among the external theories we find the well-known classical axiomatization $\kf$ of Kripke's construction \cite{fef91}, and also the iteration of reflection studied by Horsten and Leigh. Although external axiomatizations invoke principles that are not valid in the intended semantics and they usually cannot be closed under global reflection, they are usually proof-theoretically stronger than the corresponding internal axiomatizations. As a consequence, they also deem true more sentences that belong to the intended extension of the truth predicate than the corresponding internal axiomatizations; for some authors (cf. \cite{hal14,han16}), this is considered a clear advantage over internal theories. 

For instance, $\kf$ is proof-theoretically much stronger than its natural internal counterpart $\pkf$. The former proves transfinite induction for all sentences of the language with the truth predicate up to any ordinal smaller than $\varepsilon_0$, the latter only up to any ordinal smaller than $\omega^\omega$. The main result of the present article is that two steps of reflection over $\tso$ enbles us to recapture all principles of $\pkf$ and prove significantly more transfinite induction than what is available in $\pkf$. Moreover, iterated reflection on $\tso$ enables us to reach the strength of $\kf$.

\section{The core laws of truth}\label{cortru}




In this section we introduce the main components of the theory $\tso$. We first introduce a two-sided sequent version of Basic De Morgan logic and state some simple properties of this calculus. We then introduce the principles governing the objects to which truth is ascribed, which will amount to the axioms of a very weak arithmetical theory.  Finally we state the truth theoretic principles of $\tso$. 

\subsection{Basic De Morgan Logic}  We employ a two-sided sequent calculus $\bdm$ reminiscent of the one employed in \cite{hal14}; sequents are expressions of the form $\Gamma \Ra \Delta$ where $\Gamma,\Delta$ are finite sets of formulas. We write $\neg \Gamma$ for $\{\neg A\sth A\in \Gamma\}$. $\bdm$ is a subsystem of a suitable two-sided classical calculus; its axioms and rules are listed in Table \ref{table:bmdpri}. Intuitively, $\bdm$ is obtained from classical logic by replacing the usual clauses for negation with ({\sf CP1}) and ({\sf CP2}) below. However,  the general negation rules ({\sf CP1-2}) enable us to derive the sequents $A\Ra \neg\neg A$ and $\neg \neg A\Ra A$ for all formulas $A$ and make the following contraposition rule admissible in $\bdm$:
	$$
			\AxiomC{$\Gamma \Ra \Delta$}\RightLabel{$\mathsf{(Cont)}$}
			\UnaryInfC{$\neg \Delta \Ra \neg \Gamma$}
			\DisplayProof{}. 
	$$
	
The following closely related lemma will be extensively used in what follows:
	\begin{lemma}
		For a signature $\mc S=\{P_1,\ldots,P_n\}$, if for all atomic formulas $A$ of $\mc S$ we can prove $\Ra A,\neg A$, then $\bdm(\mc S)$ is closed under the following classical rules for negation:
		
		\begin{align*}
			&\AxiomC{$ \Gamma \Ra \Delta,A$} \RightLabel{$(\neg {\sf L})$}
		\UnaryInfC{$\neg A, \Gamma \Ra \Delta$}
		\DisplayProof
			&&	\AxiomC{$ A,\Gamma \Ra \Delta$} \RightLabel{$(\neg {\sf R})$}
		\UnaryInfC{$ \Gamma \Ra \Delta,\neg A$}
		\DisplayProof
		\end{align*}
	\end{lemma}
\noindent   $\bdm$ enjoys standard properties of Gentzen-type sequent calculi such as substitution, inversion, and cut elimination. 

A natural semantics for $\bdm$ is given in terms of four-valued models, that is we also allow predicates  with a partial or a paraconsistent behaviour (gaps and gluts).\footnote{Our logic is close to what is sometimes called {\sf FDE}. We follow, however, Field's terminology.}  \label{pagesat}The intended satisfaction relation has a double clause: a sequent is satisfied in a model $\mc M$ just in case if all formulas in the antecedent are true in $\mc M$ there is a formula in the consequent true in $\mc M$, {\it and} if   all formulas in the succedent are false in $\mc M$, there is a false-in-$\mc M$ formula in the antecedent. $\bdm$ is sound and complete with respect to the semantics just hinted at \cite{bla02}.

{\small	\begin{table}[h]
		\begin{tabular}{  c  c  c  }\hline\hline
			&&\\
			$A \Ra A$  &for $A\in \mc L_\T$&\\[2em]
			\AxiomC{$ \Gamma \Ra \Delta$} \RightLabel{(${\sf LW}$)}
		\UnaryInfC{$A, \Gamma \Ra \Delta$}
		\DisplayProof
				&\AxiomC{$ \Gamma \Ra \Delta$} \RightLabel{(${\sf RW}$)}
		\UnaryInfC{$ \Gamma \Ra \Delta, A$}
		\DisplayProof &\\[2em]
			\AxiomC{$\Gamma \Ra \Delta, A$} \AxiomC{$A, \Gamma \Ra \Delta$}\RightLabel{(${\sf Cut}$)}
		\BinaryInfC{$\Gamma \Ra \Delta$}
		\DisplayProof{}&&\\[2em]
		\AxiomC{$\neg \Gamma \Ra \Delta$}
			\RightLabel{($\mathsf{CP}_1$)}
			\UnaryInfC{$\neg \Delta\Ra \Gamma$}
			\DisplayProof
		&\AxiomC{$\Gamma\Ra \neg \Delta$}
			\RightLabel{(${\sf CP}_2$)}
			\UnaryInfC{$ \Delta\Ra \neg \Gamma$}
			\DisplayProof&\\[2em]
		\AxiomC{$ A, B, \Gamma \Ra \Delta $} \RightLabel{(${\sf L\wedge}$)}
		\UnaryInfC{$ A \wedge B, \Gamma \Ra \Delta$}
		\DisplayProof 
		&\AxiomC{$ \Gamma \Ra \Delta, A $}
		\AxiomC{$ \Gamma \Ra \Delta, B $}
		\RightLabel{(${\sf R\wedge}$)}
		\BinaryInfC{$ \Gamma \Ra \Delta, A \wedge B $}
		\DisplayProof
			&\\[2em]
		\AxiomC{$ A, \Gamma \Ra \Delta $}
		\AxiomC{$ B, \Gamma \Ra \Delta $}
		\RightLabel{(${\sf L\vee}$)}
		\BinaryInfC{$ A \vee B, \Gamma \Ra \Delta$}
		\DisplayProof
			&\AxiomC{$ \Gamma \Ra \Delta, A, B $} \RightLabel{(${\sf R \vee}$)}
		\UnaryInfC{$ \Gamma \Ra \Delta, A \vee B$}
		\DisplayProof&\\[2em]
		\AxiomC{$\Gamma, A(t) \Ra \Delta$}\RightLabel{(${\sf L \forall}$)}
        \UnaryInfC{$\Gamma, \forall x A \Ra \Delta$} 
        \DisplayProof{}		
		&\AxiomC{$\Gamma \Ra \Delta, A(t)$}\RightLabel{(${\sf R \exists}$)} 
		\UnaryInfC{$\Gamma \Ra \Delta, \exists x A$}
		\DisplayProof {}
		&\\[2em]
		\AxiomC{$\Gamma \Ra \Delta, A(x)$} 
		\RightLabel{(${\sf R \forall}$)}
		\UnaryInfC{$\Gamma \Ra \Delta, \forall x A$}
		\DisplayProof {}
		&\AxiomC{$\Gamma,A(x) \Ra \Delta$} 
		\RightLabel{(${\sf L \exists}$)}
		\UnaryInfC{$\Gamma, \exists x A \Ra \Delta$}
		\DisplayProof &\\[1em]
		$x$ not free in $\Gamma, \Delta$&$x$ not free in $\Gamma, \Delta$&\\[2em]\hline\hline\\[-0.5em]
	\end{tabular}\caption{The system $\bdm$}\label{table:bmdpri}
	\end{table}}

%
%
%
%
%
%
%

\subsection{The theory $\tso$}

To formulate $\tso$, we first consider identity, which is governed by usual principles:
	\begin{align}
		\tag{{\sf Id1}} &\Ra t=t\\
		 \tag{{\sf Id2}}& s=t,\,A(s)\Ra A(t)
	\end{align}
$\tso$ and all its extensions will be formulated in the language of arithmetic $\mc L$, expanded with finitely many function symbols corresponding to suitable elementary operations and the truth predicate $\T$. We call the resulting language $\lt$. $\tso$ will also contain initial sequents $\Ra A$ for all basic axioms $A$ of a suitable system of arithmetic, in our case Kalmar's elementary arithmetic ${\sf EA}$ formulated in $\mc L_\T$ (cf. \cite{haj93,bek05}). In addition, our basic theory features an induction rule 
\begin{equation}
\tag{$\Delta_0$-$ { \sf IND }$} \AxiomC{$\Gamma,A(x)\Ra A(x+1),\Delta$}
\UnaryInfC{$\Gamma,A(0)\Ra A(t),\Delta$}
\DisplayProof
\end{equation}
for $x$ not free in $A(0),\Delta,\Gamma$, $t$ is arbitrary, and $A$ is a $\Delta_0$-formula of the language $\mc L$ of arithmetic without the truth predicate.  We call the resulting system ${\sf Basic}$. 


The core principles of truth capture the fundamental idea that one is justified in asserting a sentence $A$ precisely when she is justified in asserting that $A$ is true.
\begin{dfn}[The system $\tsn$]
$\tsn$ is obtained by extending $\basic$ with the initial sequents
\begin{align*}
\tag{$\T1$}\T(\ulc A\urc) &\Ra A\\[1ex]
\tag{$\T 2$}A&\Ra \T(\ulc A\urc)
\end{align*}
for all $\mc L_\T$-sentences $A$.
\end{dfn} 
\noindent $\tso$ stands for `truth sequents'. The subscript $0$ indicates a restriction of induction to $\Delta_0$ formulas; its absence indicates full induction.  The semantic conservativeness -- and therefore the consistency -- of $\tso$ over $\basic$ can be obtained by expanding any model of the latter with an interpretation of the truth predicate resulting from a positive inductive definition along the lines of the Kripke construction (cf. \cite[\S5]{can89}).  The following observation can be found in \cite[Lem. 16]{hah06}. 
\begin{lemma}
$\Ra A, \neg A$ is derivable in $\mathsf{TS}_0$ given that $A$ is arithmetical.
\end{lemma}

The principles of $\tso$ are therefore just right to capture the desired assertability conditions for truth ascriptions: its basic truth-theoretic principles (\T1)-(\T2) are in fact not as strong as the classical Tarski-biconditionals: otherwise they would lead to inconsistency. But they are also stronger than mere inference rules, as the latter do not allow for conditionalization for arithmetical sentences. 

In addition, we can think of $\tso$ as a minimal internal axiomatization of a fixed-point construction along the lines of Kripke's \cite{kri75}. The fixed-points we are interested in are in fact fixed-points of a monotone operator $\Gamma$ associated with the Basic De Morgan evaluation scheme.\footnote{For a definition of the operator and the evaluation scheme we refer the reader to Halbach \cite[section 15.1]{hal14}.} The crucial property of Kripke style fixed points $S$, i.e.~sets of sentences $S$ such that $\Gamma(S) = S$, is that every sentence $A$ is in $S$ iff ${ \sf T} ( \ulcorner A \urcorner)$ is in $S$. By combining this fact with the notion of satisfaction introduced on p.\pageref{pagesat} we can easily see that for a fixed point $S$, the model $(\mathbb{N}, S)$ satisfies $\tso$ when $S$ is taken to be the extension of the truth predicate. Moreover, for $(\mathbb{N}, S)$ to satisfy $\tso$, $S$ has to contain the same sentences as $\Gamma(S)$. This means that $S$ is a fixed-point of $\Gamma$ iff $(\mathbb{N}, S)$ satisfies $\tso$, and therefore it is an internal axiomatization of the fixed-point construction in the sense of \S1. Moreover any $\mathcal{L}_{\sf T}$-theory in ${\sf BDM}$ satisfying the adequacy condition just considered will contain the principles of $\tso$. 


\subsection{The weakness of $\tso$ and the advantages of reflection in ${ \sf BDM}$}

In the previous subsection we have introduced $\tso$ as a natural and simple theory capturing distinctive features of the notion of truth. However, the theory in itself falls short of several adequacy requirements generally imposed to theories of truth--see for instance \cite{lei07} and \cite{hah15}. For instance one important requirement for truth is compositionality, which explains how we can understand complex sentences only on the basis of an understanding of its compounds and its logical structure. $\tso$ is clearly not compositional as can be realized by considering the quantifiers. For a universally quantified sentence, such as $\forall x A(x)$, the theory $\tso$ is not able to derive in general sequents explaining how the truth value depends on the truth values of its compounds $A( t)$ in the sense that $\forall x  {\sf T} ( [ A x]) \Ra { \sf T } ( \ulcorner \forall x  A (x)\urcorner )$.
 
Another criterion of adequacy for theories of truth that we want to impose is the ability to prove important generalizations; one especially desirable generalization to follow from a theory of truth is the soundness of the base theory stated in the form of the global reflection principle for $\basic$, i.e.~
\[ \tag{${ \sf GRF_\basic}$}  {\sf Bew}_\basic ( x ) \Rightarrow \T (x), \] which is not derivable in $\tso$ (where  ${\sf Bew}_T(x) $ is a canonical provability predicate for $T$). The underivability of the global reflection principle directly follows from the fact that $\tso$ is a conservative extension of $\basic$. 

It is difficult to give an exact criterion of what it means for a truth theory to prove all the important generalizations that it should prove. But it is obvious that a stronger theory of truth is, all other things being equal, better than a weaker one. It is therefore natural to aim for a theory of truth that is, relative to a base theory, as strong as possible. Taken as a measure of the proof-theoretic strength of $\tso$, the conservativeness of $\tso$ over $\basic$ only considers arithmetical sentences, but we want our measure also to take into account generalizations involving the truth predicate. Therefore we will mainly focus on the amount of transfinite induction for the language $\mathcal{L}_{ \sf T}$ we can prove. 

Following a strategy already proposed and defended in  \cite{hol15} for theories formulated in classical logic, one may think of $\tso$ as {\it implicitly}  containing stronger principles, including compositional ones and principles of transfinite induction. This relation of implicit containment can be unfolded via postulating a hierarchy of reflection principles over $\tso$. Traditionally, reflection principles for a theory $T$ are explicit soundness assertions (``\emph{whatever is provable in $T$, is true}''). The soundness of $T$ is naturally expressed via ${\sf GRF}_T$.\footnote{This reading of reflection is ubiquitous in the literature. See for instance the classical handbook entry \cite{smo77} and \cite{hal14}. Kreisel and L\'evy in \cite{krl68} clearly states that global reflection is the intended soundness claim for a theory $T$.}
However, by Tarski's undefinability theorem, ${\sf GRF}_T$ can only be formulated if the expressive resources of $T$ are increased with a fresh truth predicate. Therefore, if one wants to express soundness in an arithmetical language, one must resort to schemata. A well-known candidate is what is widely known as the {\it uniform reflection principle} for $T$:
\[
\tag{${\sf RFN}_T$} \forall x({\sf Bew}_T([ A(x)]) \ra A(x))
\]
where $[A x] := {\sf sub} ( \ulc A x \urc, \ulc x \urc, {\sf num}(x))$, where ${\sf num}$ represents the elementary function sending a number $n$ to the $n$-th numeral and  ${\sf sub}$ the usual substitution function, whereas $x(y/v)$ will stand for ${\sf sub}(x,v,{\sf num}(y))$ with $v$ coding a variable free in $x$. ${\sf RFN}_T$ states that, for every number $x$, if $A$ is satisfied by the numeral for $x$, provably in $T$, then $A$ is satisfied by $x$. However, we are mainly concerned with languages that {\it do} contain a truth predicate. In this context, therefore, the most natural way to express the soundness of a theory is by means of global reflection. In fact, if the truth predicate satisfies minimal conditions, the global reflection principles implies all instances of uniform reflection for a theory $T$. 

There is therefore an intuitive connection between uniform and global reflection: both are intended to express the soundness of the base theory. It turns out, however, that this connection is lost in the classical axiomatizations of Kripke's fixed point construction considered by Leigh and Horsten in \cite{hol15}.  For $T$ an axiomatization of Kripke's fixed point construction in classical logic, in fact, the result of adding {\sf GRF}$_T$ to it determines a severe restriction of the class of acceptable models: all \emph{consistent} fixed points are excluded, i.e., if $(\mbb N,S)$ models $T+ {\sf GRF}_T$ with $S$ a fixed point, then $S$ is inconsistent.\footnote{ By the diagonal lemma, the arithmetical  part of $T$ already proves $(\lambda\land \neg \T\ulc \lambda\urc) \vee(\neg \lambda\land \T\ulc \lambda\urc)$ for $\lambda$ a liar sentence. Therefore $T+ {\sf GRF}_T$ proves $\T\big (\ulc (\lambda\land \neg \T\ulc \lambda\urc) \vee(\neg \lambda\land \T\ulc \lambda\urc)\urc)$. Since $T$ is an axiomatization of the class of Kripke fixed points, we can use compositional and truth-iteration principles to obtain, still in  $T+ {\sf GRF}_T$, $\T(\ulc \lambda \land \neg \lambda\urc)$. A well-known example of such a theory $T$ is a the theory ${\sf KF}$ from \cite{fef91} -- see also \cite{hal14}.} By contrast, $T+ {\sf RFN}_T$ can have models of the form $(\mbb N, S)$ for $S$ a consistent fixed point (in fact all consistent fixed points). 

There is a natural explanation for the internal inconsistency of $T+{\sf GRF}_T$: classical theories $T$ of the sort just mentioned are in fact {\it unsound} with respect to the notion of truth captured by $T$, and  ${\sf GRF}_T$ makes this explicit.  In fact, many theorems involving the truth predicate in a classical axiomatization $T$ of the fixed-point construction are outside the extension of the truth predicate given by consistent fixed points. The classical tautology $\lambda \vee\neg \lambda$ involving a liar sentence is one such example. Uniform reflection alone, in theories such as $T$, does not suffice to uncover their unsoundness:\footnote{This is also the reason why Horsten and Leigh could consider iterations of uniform reflection without restrictions on the fixed-points models.} this is the sense in which the intimate connection between global and uniform reflection is lost in the classical setting.

The close connections between the two forms of reflection just considered, however, can be restored by moving to internal axiomatizations of Kripke fixed points such as extensions of $\tso$. To see this, we first reformulate both principles in rule form and adapt them to the sequent-style formulation of $\tso$ we have chosen. 
\begin{align*}
&\AxiomC{$\Ra { \sf Bew}_T ( [ A {x} ] )$}\RightLabel{({$ { \sf RFN}_T^R $})}
\UnaryInfC{$\Ra A(x)$}\DisplayProof
&&\AxiomC{$\Ra {\sf Sent}_{\lt}(x)\land { \sf Bew}_T (x)$}\RightLabel{({$ { \sf GRF}_T^R $})}
\UnaryInfC{$\Ra\T x$}\DisplayProof
\end{align*}
In ({$ { \sf RFN}_T^R $}), $A(x)$ is a formula of $\lt$ with one free variable, in $ { \sf GRF}_T^R $ the elementary predicate  ${\sf Sent}_{\lt}(x)$ expresses the set of $\lt$-sentences, and in both rules ${ \sf Bew}_T(A)$ states that the sequent $\Ra (A)$ is derivable in $T$. Finally, we introduce an extension of $\tso$ obtained by replacing the axioms (\T1) and (\T2) with   
\begin{itemize}\label{uts}
 			\item[(i)] $ A ( x ) \Ra \T [ A x ] $;
 			 \item[(ii)] $ \T [ A x ] \Ra A(x)$.
	\end{itemize}
\noindent We call the resulting system $\utso$  (``\emph{uniform} $\tso$). We can now establish that not only uniform and global reflection are connected in Basic De Morgan logic, but that they actually {\it coincide}.
\begin{proposition}\label{grun}
 Let $T$ contain $\utso$. Then $T+{ \sf RFN}_T^R$ and $T+{ \sf GRF}_T^R$ are identical theories. 
\end{proposition}
\begin{proof}
We start by showing that global reflection entails uniform reflection. Reasoning in $T+{ \sf GRF}_T^R$, we assume that the sequent $\Ra {\sf Bew}_T([Ax])$ is derivable in it. Then, by ${ \sf GRF}_T^R$, we have $\T[Ax]$ and therefore $Ax$ by (ii) above. 

For the other direction, we reason in $T+{ \sf RFN}_T^R$ and assume that the sequent $\Ra {\sf Sent}_{\lt}(x)\land {\sf Bew}_T(x)$ is derivable in it. Also, we know that (i) and (ii) are derivable sequents of $\utso$ -- and then also of $T$ --, and therefore there will be a canonical provability predicate for sequents ${\sf Pr}_{T}(x,y)$ for $T$ such that $\Ra {\sf Pr}_{T}(x,[\T x])$ is a derivable sequent of $T$. By combining this latter fact with our assumption, we obtain $\Ra {\sf Bew}_{T}([\T x])$. By ${ \sf RFN}_T^R$, therefore, we can conclude $\Ra \T x$, as desired. 
\end{proof}

Proposition \ref{grun} suggests that Leigh and Horsten's project can be more coherently carried in the context of \emph{nonclassical} theories of truth. In the next section we will in fact employ strengthenings of the reflection principles considered in Proposition \ref{grun} to unfold the truth-theoretic and mathematical content implicit in the acceptance of $\tso$.


\section{Reflecting on $\tso$}

This section introduces the main results of the paper: in \S3.1 we discuss several alternative reflection rules and motivate the choice of a particular form of reflection on admissible rules that turns out to be stronger than simple reflection on derivable sequents. In \S3.2 we show that the closure under two applications of our rule of reflection suffices to recover the strong internal axiomatization of Kripke's fixed point $\pkf$. Finally, in \S3.3 we show that the result of reflecting twice on $\tso$ proves more transfinite induction for the language with the truth predicate than $\pkf$ itself. We conclude the section by investigating further iterations of reflection.


\subsection{Reflection on sequents and rules}
\label{prelim}

In what follows, we assume a canonical G\"odel numbering for $\lt$-expressions. For a fixed expression $e$ of $\lt$,  we will use the usual G\"odel corners for the closed term of $\lt$ representing Gödel number $\#e$ of $e$.  Therefore, for formulas $A$ of $\lt$, we will have  $\ulcorner A \urcorner = \overline{ \# A}$. Similarly, for sequents $\Gamma \Ra \Delta$,  $\ulcorner \Gamma \Ra \Delta \urcorner = \overline{ \# \Gamma \Ra \Delta}$, where the G\"odel code of $\Gamma \Ra \Delta$ is taken to be an ordered pair whose components are the codes of the finite sets $\Gamma$ and $\Delta$.\footnote{We assume that the code of the finite set $\Gamma$ is the code of the sequence of codes of formulas in $\Gamma$ in ascending order.}  Closed terms standing for specific G\"odel codes of $\lt$-expressions contrast with open terms standing for templates to generate such closed terms: a well-known example of such a template is the open $\lt$-term ${ \sf sub } (\ulc A(v)\urc,\ulc v\urc,{\sf num}(x))$, standing for the result of formally substituting, in the formula $A(v)$, the free variable $v$ with the numeral for $x$. To distinguish these open terms from specific codes, we use square brackets instead of G\"odel codes, so that, for instance,   $[A(x)]$ stands for ${ \sf sub } (\ulc A(v)\urc,\ulc v\urc,{\sf num}(x))$.\footnote{The square brackets notation is often replaced by the so-called Feferman dot notation, in which, for instance, ${ \sf sub } (\ulc A(v)\urc,\ulc v\urc,{\sf num}(x))$ is abbreviated with $\ulc A(\dot x)\urc$. }   This distinction clearly generalizes to sequents and formulas with more than one free variable: $[ \Gamma \vec{x} \Ra \Delta \vec{x}]$ refers to the simultaneous substitution in $\ulc \Gamma \Ra \Delta\urc$ of the variables in the strings $\vec x$ with their corresponding numerals, where of course $[ \Gamma x \Ra \Delta x]$ is short for ${ \sf sub } (( \ulcorner \Gamma \urcorner , \ulcorner \Delta \urcorner),\ulc x\urc, { \sf num } (x) )$. When it is clear from the context which free variable we are formally substituting, we will omit it and treat ${\sf sub}$ as a binary function. 

 As we have seen, in the classical setting the uniform reflection schema and rule take the form:
\begin{align*}
\tag*{(${ \sf RFN}_T^R)$}  &{ \sf Pr}_T ( [ A ({x}) ] ) \Ra A(x) \\[1em]
\tag{$ { \sf URFN}_T^R $}&\AxiomC{$\Ra { \sf Pr}_T ( [ A ({x}) ] )$}
\UnaryInfC{$\Ra A(x)$}
\DisplayProof
\end{align*}
\noindent Over ${\sf EA}$,  $ { \sf URFN}_T^R $ and ${ \sf RFN}_T^R$ are equivalent, as shown by Feferman in \cite{fef60}. In the non-trivial direction, i.e. going from the rule to the initial sequent, one shows that $\basic$ suffices to formalize the fact that the sentence ${\sf Prf}_T(\bar n, \ulc A(\bar m)\urc)\ra A(\bar m)$ is provable in $T$ for any $m,n\in \omega$. Therefore one application of $({ \sf URFN}_T^R)$ yields (${ \sf RFN}_T^R)$.

In the nonclassical setting the situation is different. Whereas in the classical setting we can formulate $({ \sf URFN}_T^R)$ and (${ \sf RFN}_T^R)$ in a one-sided sequent calculus, there are good reasons to stick with a two-sided calculus for Basic De Morgan logic. In a one-sided classical system, in fact, sequents $A,\neg A$ play the role that initial sequents $A\Ra A$ play in a two-sided setting. In our system this correspondence breaks down: first of all, $A,\neg A$ is not generally valid in our intended semantics -- if $A$ is a liar sentence, for instance --, whereas  $A\Ra A$ are initial sequents of our system.   Moreover, there is no conditional naturally corresponding to the sequent arrow since $\Ra A \rightarrow A$ is just a notational variant of $\Ra A\vee\neg A$. 

We therefore opt for a formulation of our first reflection principle as applying to provable (two-sided) sequents. As a consequence, basic syntactic considerations force us to formulate reflection in rule-form. The \emph{uniform reflection principle for sequents} of $T$ takes the following form:
\[
\tag{${\sf r}_T$}\label{refrul}
\AxiomC{$\Ra \mathsf{Pr}_T ( [\,\Gamma  \vec{x} \Ra \Delta \vec{y}\,])$}
\UnaryInfC{$\Gamma \Ra \Delta$}
\DisplayProof
\]

\noindent 
But the simple rule of reflection (${\sf r}_T$) is not the only form of reflection that will be relevant for what follows. A suitable conditional -- such as the classical or the intuitionistic conditional -- enables one to compress in one sequent chains of reasoning featuring embedded implications. In our setting, the highly \emph{meta-theoretic} nature of the sequent arrow forces us to capture these chains of reasoning explicitly via suitable extensions of the simple reflection rule (${\sf r}_T$).  One way to achieve this is to focus not only on provable sequents, but also to take into account   rules admissible in $T$.  
 
In order to introduce these generalized forms of reflection  we  consider a variation of the reflection rule based on a two place provability predicate $\pr^2_{T} (x,y) $ understood as `the inference from the sequent $x$ to the sequent $y$ is {\it provably} admissible in $T$'. In addition to being admissible, a {\it provably admissible rule} in a theory $T$ requires the existence of a $T$-provable proof transformation of the proof of the premise into a proof of the conclusion of the rule.\footnote{For instance, although the rule of cut applied to `geometric' formulations of Robinson's arithmetic {\sf Q} is admissible in it, cut is not provably admissible in {\sf Q} as this procedure is of hyperexponential growth rate. (For a geometric presentation of Robinson's arithmetic and for the cut elimination for it, see \cite{nevo98}.)}    

$\prt_T$ enjoys generalized versions of some of the properties usually ascribed to provability predicates:
	\begin{align*}
	\label{pr1}\tag{$\pr1$}& \text{If }\;\;\;
									\AxiomC{$\Gamma(x) \Ra \Delta(x)$}
									\UnaryInfC{$\Theta(x)\Ra \Lambda(x)$}
									\DisplayProof\;\;\;
							\text{is admissible in $T$, provably in $\basic $, then}\\[1ex]
				 & \basic \vdash \prt_T([\Gamma(x)\Ra \Delta(x)],[\Theta(x)\Ra \Lambda(x)])\\[1em]
		\tag{$\pr2$}&	\text{If the sequents}\\[1ex]
				&\hspace{10pt}\prt_T([\Gamma(x)\Ra \Delta(x)],[\Theta(x)\Ra \Lambda(x)]); \\
				&\text{and }\\
				&\hspace{10pt}\pr_T([\Gamma(x)\Ra \Delta(x)])\\[1ex]
				&\text{are derivable in $\basic$, then also }\\[1ex]
				&\hspace{10pt}\pr_T([\Theta(x)\Ra \Lambda(x)])\\[1ex]
				&\text{is derivable in $\basic$}.
	\end{align*}
\noindent We can then define the \emph{uniform reflection principle for provably admissible rules} in $T$:

	\begin{align*}
	\tag{${\sf R}_T$} 
		&			\AxiomC{$\prt_T([\Gamma(x)\Ra \Delta(x)],[\Theta(x)\Ra \Lambda(x)])$}
					\AxiomC{$\Gamma(x) \Ra \Delta(x)$}
					\BinaryInfC{$\Theta(x)\Ra \Lambda(x)$}
					\DisplayProof
	\end{align*}

\noindent Obviously,  in the context of any reasonable theory $T$, ${\sf R}_T$ implies ${\sf r}_T$.  If $T$ is an axiomatizable theory, then the \emph{reflection on $T$} is the closure of ${ \sf Basic }$ under the reflection rules $\rr{T}$ and $\fr{T}$:

	\begin{align*} 
		\rr{T} &:= \basic+({\sf r}_T)\\
		\fr{T}&:=\basic+({\sf R}_T)
	\end{align*}

\noindent Theories obtained by iterating our reflection rules are then defined in a standard manner: for instance,  $\fr{\fr{T}}$ is the result of closing $\fr{T}$ under ${\sf R}_{{\sf R}(T)}$. We abbreviate $\fr{\fr{T}}$ as ${ \sf R }^2 ({T})$, and similarly for more iterations. 

\label{questrul}We have introduced (${\sf R}_T$) as a generalization of $({\sf r}_T)$. A natural question is whether (${\sf R}_T$) is actually stronger than the simpler rule. We will not answer to this question in this paper but we will prove some facts that may be relevant for a future answer. For instance, we now provide an upper bound for the strength of ${\sf r}(\utso)$; later -- cf. Proposition \ref{pr2ti} -- we will show that the resulting theory is a proper subtheory of ${\sf R}^2(\tso)$.

 The upper bound for ${\sf r}(\utso)$ that we now provide is given in terms of the theory $\pkf$ that was mentioned in the introduction. $\pkf$ is also formulated in the language $\lt$, and its axioms and rules are displayed in Table \ref{table:pkf}. 
 
 \begin{table}
 	\begin{tabular}{ l | l }\hline\hline\\
		\textsc{logic}&   logical initial sequents and rules of $\bdm$\\[5pt]
		\textsc{identity}& {\sf Id1}, {\sf Id2}\\[10pt]
		\textsc{arithmetic}& $\Ra A$ for $A$ a basic axiom of {\sf EA}\\[5pt]
		& plus the {\it full} induction rule for $\lt$:\\[8pt]
		& \AxiomC{$\Gamma,A(x)\Ra A(x+\bar 1),\Delta$}\RightLabel{({\sf IND})}
\UnaryInfC{$\Gamma,A(0)\Ra A(t),\Delta$}
\DisplayProof\\[20pt]
\textsc{atomic truth}& ($\T=_1$)	 \hspace{10pt}${ \sf ct } ( x ) , { \sf ct } ( y ), { \sf val} (x) = {\sf val} (y) \Ra \T ( x \subdot{=} y ) $ \\[5pt]
&($\T=_2$)  \hspace{10pt}${ \sf ct } ( x ) , { \sf ct } ( y ),   \T ( x \subdot{=} y )  \Ra { \sf val} (x) = {\sf val} (y)$\\[5pt]
&($\T\T_1$)\hspace{10pt} $\T [\T x] \Ra \T x$\\[5pt]
	&($\T\T_2$)\hspace{10pt}	$\T  x \Ra  \T [\T x]$\\[10pt]
	{\sc truth principles}&($\T \wedge_1$)\hspace{10pt}$\sentt (x \subdot\wedge y), \T (x) \wedge \T (y) \Ra \T ( x \subdot\wedge y )$\\[5pt]
	{\sc for connectives}&($\T \wedge_2$)\hspace{10pt}$\sentt (x \subdot\wedge y),  \T ( x \subdot\wedge y ) \Ra \T (x) \wedge \T (y)$\\[5pt]
	&($\T \vee_1$)\hspace{10pt}$\sentt (x \subdot\vee  y), \T (x) \vee \T (y) \Ra \T ( x \subdot\wedge y )$\\[5pt]
	&($\T \vee_2$)\hspace{10pt}$\sentt (x \subdot\vee  y),  \T ( x \subdot\vee  y ) \Ra \T (x) \vee  \T (y) $\\[5pt]
	&($\T \neg_1$)\hspace{10pt}$\sentt (x), \T (\subdot\neg x) \Ra \neg \T ( x )$\\[5pt]
	&($\T \neg_1$)\hspace{10pt}$\sentt  (x), \neg \T ( x ) \Ra \T (\subdot\neg x) $\\[10pt]
	{\sc truth principles}&($\T\forall_1$)\hspace{10pt}$\sentt (\subdot\forall yx), \forall y \T x(y/v) \Ra \T ( \subdot\forall y x )$\\[5pt]
	{\sc for quantifiers}&($\T\forall_2$)\hspace{10pt}$\sentt (\subdot\forall y x), \T ( \subdot\forall y x ) \Ra \forall y \T x(y/v)$\\[5pt]
	&($\T\exists_1$)\hspace{10pt}$ \sentt (\subdot\exists y x), \exists y \T x(y/v)  \Ra \T ( \subdot\exists y x )$\\[5pt]
	&($\T\exists_2$)\hspace{10pt}$\sentt (\subdot\exists y x),\T ( \subdot\exists y x )\Ra \exists y \T x(y/v)$\\[5pt]\hline\hline
	\end{tabular}\caption{The theory $\pkf$}\label{table:pkf}
 \end{table}
 
  In \T=$_{1-2}$, the function symbol $\subdot{=}$ represents the elementary syntactic operation  of forming an identity statement out of two terms. A similar notation will be applied for other syntactic operations.  As mentioned earlier, $\pkf$ is an internal axiomatization of Kripke's theory of truth. Crucially, $\pkf $ is fully compositional as also negation commutes with the truth predicate. Halbach and Horsten in \cite{hah06} have measured the proof-theoretic strength of $\pkf$ by showing that  $\pkf$ proves arithmetical transfinite induction up to the ordinal $\vphi_\omega0$. Therefore we can use $\pkf$ as means of comparison for our theories of iterated reflection. 

	\begin{proposition}\label{upper}
		${ \sf r (UTS_0) }$ is a subtheory of ${ \sf PKF }$.
	\end{proposition}
\begin{proof} To prove Proposition \ref{upper} we only need to check that $\pkf$ can handle reflection. Therefore we first establish that $\pkf$ is strong enough to prove the soundness of $\utso$. To this end, for ordinal codes $\alpha$, we define a hierarchy of predicates ${\sf Tr}_\alpha(\cdot)$  as $\T(\cdot)\land {\sf Sent}_{\mc L_\T^{<\alpha}}(\cdot)$.\footnote{The truth predicates ${\sf Tr}_\alpha$ can be defined for as many ordinals as we can code in our theory. In \S\ref{subse:trare}, in particular, we will employ a coding for ordinals  smaller than $\Gamma_0$. } Halbach and Horsten establish in \cite{hah06} that $\pkf$ proves the predicates ${\sf Tr}_\beta$ to behave like Tarskian truth predicates for $\beta<\omega^\omega$: that is for formulas of $\lt$ that are in $\lt^{<\omega}$, the classical commutation conditions for typed truth predicates hold, while for $\lt$-formulas $A$ not in $\lt^{\omega}$, we can prove $\neg \T_\omega \ulc A\urc$.

We can extend this definition to sequents via the predicate ${\sf TR}_\alpha(\cdot)$ in the following way:
	\begin{align*}
		{\sf TR}_\alpha(\ulc \Gamma\Ra \Delta\urc)\;:\lra\; & \big(({\sf Tr}_\alpha(\ulc \bigWedge\Gamma\urc )\ra {\sf Tr}_\alpha(\ulc \bigVee\Delta\urc)) \;\land\\
			& ({\sf Tr}_\alpha(\ulc \neg\bigVee\Delta\urc )\ra {\sf Tr}_\alpha(\ulc\neg  \bigWedge\Gamma\urc)\big)
	\end{align*}
	From the proof-theoretic analysis of $\pkf$ in \cite{hah06} we know that the  truth predicates up to $\omega^\omega$ behave classically in it, and that we can employ the material conditional to carry out the inductive proof of the following:
	\begin{equation}
		\pkf \vdash\Ra  {\sf Pr}_{\sf UTS_0}([\Gamma x\Ra \Delta x])\ra {\sf TR}_\omega([\Gamma x \Ra \Delta x])
	\end{equation}
	
	The proof employs the induction rule of $\pkf$. 
	It suffices, therefore, to establish
	\begin{align}
		\label{na1}& \Ra{\sf Prf}_{\utso}(0,[\Gamma x \Ra \Delta x])\ra {\sf TR}_\omega( [ \Gamma x \Ra \Delta x])\\
		\label{na2}&{\sf Prf}_{\utso}(u,[\Gamma x \Ra \Delta x])\ra {\sf TR}_{\omega}([\Gamma x \Ra \Delta x])\Ra \\ \notag &{\sf Prf}_{\utso}(u+1,[\Gamma x \Ra \Delta x])\ra {\sf TR}_{\omega}[\Gamma x \Ra \Delta x])
	\end{align} 
where ${\sf Prf}_{T} (x,y)$ expresses that $y$ is provable in $T$ with a proof of length less or equal to $x$.

We consider the crucial case of the characterizing principles of $\utso$, (i)-(ii) on page \pageref{uts}. Reasoning classically in $\pkf$, we assume 
\begin{equation}\label{ubpr1}
	{\sf Prf}_{\utso}(0,[\T[Ax] \Ra  Ax])
\end{equation}
We need to show
\begin{align}
	\label{ubpr2}& \T_\omega[\T[Ax]]\ra \T_\omega[Ax]\\
	\label{ubpr3}& \T_\omega [\neg Ax]\ra \T_\omega [\neg \T[Ax]]
\end{align}
We start with \eqref{ubpr2}. If $\T_\omega[\T[Ax]]$, then for some $n\in \omega$ and $m<n$, $\T_n[\T_m[Ax]]$. Therefore, since $\T_n$ and $\T_m$ are Tarskian truth predicates, also $\T_n[Ax]$. 

Similarly for \eqref{ubpr3}, if $\T_\omega [\neg Ax]$, then $\T_n[\neg Ax]$ for some $n\in \omega$, and, since $\T_n[Ax]$ is in $\lt^{n+1}$, also $\T_{n+1}\neg[\T_n[Ax]]$ and therefore $\T_\omega[\neg \T[Ax]]$.
	\end{proof}

\subsection{Recovering compositionality by reflection}

One of the goals of this section is to show that by reflecting on our core laws of truth we can recover desirable compositional principles. More specifically, reflecting on $\tso$ is sufficient to recover the initial sequents  and the full induction rule of  ${ \sf PKF }$. 

In a first step we show that adding the reflection principle for $\mathsf{TS}_0$ to ${ \sf Basic }$ allows us to derive the initial sequents of $\utso$.
\begin{lemma} $ \mathsf{UTS}_0 \subseteq \mathsf{r} (\mathsf{TS}_0)$.
\end{lemma}
\begin{proof}
 For all $n \in \omega$ we have $\mathsf{TS}_0 \vdash \T ( \ulcorner A(n) \urcorner )\Ra A ( n ) $.
Therefore, ${ \sf Basic}$ proves:
	\begin{equation}\label{einz}
	 	\forall y ( \sentt(y) \rightarrow {\sf Ax}_{\tso} ( {\sf sub}(\ulcorner \T x \urcorner, {\sf num}(y)), y )) 
	\end{equation}
and 
	\begin{equation}\label{zwei}
		 \forall x\,( {\sent}( [Bx] ),
	\end{equation}
where, we recall, $[Bx]:={\sf sub}( \ulcorner Bv \urcorner ,{\sf num}(x) )$ for all $\lt$-formulas $B$ with one free variable.
Therefore, by combining \eqref{einz} and \eqref{zwei}, we also have in $\basic$
	\begin{equation}
		{\sf Ax}_{\tso} ( {\sf sub}( \ulcorner \T x \urcorner, {\sf num}([Ax]), [Ax] ) )
	\end{equation}
Therefore, by definition of the canonical provability predicate $\pr_{\tso}$,
	\begin{equation}\label{drei}
		  \pr_{\tso} ( {\sf sub}( \ulcorner \T x \urcorner , {\sf num}([Ax])), [Ax])
		 \end{equation}
Let ${\sf tr}(x)$ the elementary function that formally prefixes $\ulc\T\urc$ to the numeral for $x$. Then $\basic$ also proves the equation:
	\begin{align}
		 {\sf sub}( \ulcorner \T x \urcorner , {\sf num}( [A x]) )
 		&= {\sf sub}({\sf tr}([A v]), {\sf num}(x) )
	\end{align}
By performing the appropriate substitution in \eqref{drei}, we have
 	\begin{equation}\label{vier}
		\pr_{\tso} ({\sf sub}({\sf tr}([ Av] ), { \sf num} (x)), [ Ax ]  )
	\end{equation}
and therefore, by the properties of substitution, also
 	\begin{equation}
		\pr_{\tso}( {\sf sub}({\sf tr}( [ Av]) ,  \ulcorner Av \urcorner ), {\sf num}(x) ).
  	\end{equation}
\noindent The other direction is analogous.

In ${ \sf r } ( { \sf TS }_0)$ -- and a fortiori in {\sf R}($\tso$) -- therefore, we obtain
	\begin{align*}
		&\T[Ax]\Ra A(x)\\
		&A(x)\Ra \T[Ax]
	\end{align*}
as desired. 
\end{proof}

As a consequence of the previous lemma, in ${ \sf r } ( { \sf TS }_0)$ we can already prove the full truth sequents for atomic arithmetical formulas and for truth ascriptions containing free variables ($\T=_1$), ($\T=_2$), ($\T\T_1$), and $(\T\T_2)$. That ($\T\T_1$) and $(\T\T_2)$ are direct instances of the initial truth sequents of $\utso$ is immediate. For the identity initial sequents, a slightly more general version of the Tarski sequents would be required, namely one in which at least two free variables appear. However, since we are working over $\basic$, we can always assume that the free variable in the truth sequents of $\utso$ stands for (the code of) a string of free variables of finite length. 
\noindent
 However, also the other initial, compositional sequents of $\pkf$ for the propositional connectives $\wedge, \vee, \neg$ can be proved in ${ \sf r } ( { \sf TS }_0)$:
\begin{lemma}
In ${ \sf r } ( { \sf TS }_0)$ we can derive $(\T \wedge_{1\text{-}2})$, $(\T \vee_{1\text{-}2})$, $(\T \neg_{1\text{-}2})$.
 \end{lemma} 
\begin{proof} In $\tso$ we can directly prove the schematic form of the compositional clauses, for example $\tso \vdash \T (\ulcorner A \urcorner) \wedge \T (\ulcorner B \urcorner) \Ra (\T \ulcorner A \wedge B \urcorner) $ for all $\lt$-sentences $A,B$. By formalizing this fact in $\basic$, we obtain
\begin{equation}\label{comp1}
	\pr_{\tso} ([ \sentt ( x \subdot\wedge y ), \T x \wedge \T y \Ra \T (x \subdot\wedge y)]). 
\end{equation} 
In ${\sf r}(\utso)$, therefore, we can then move from the formalization to the full quantifiable statement 
\begin{equation}\label{comp2}
\sentt ( x \subdot\wedge y ),\T x \wedge \T y \Ra \T (x \subdot\wedge y)
\end{equation}
as desired.  The case for the other connectives are analogous. 
\end{proof}

However, by looking at Table \ref{table:pkf} one realizes that compositional initial sequents for the propositional connectives by themselves are not enough to capture all truth principles of ${ \sf PKF }$: we also need initial sequents for quantifiers and full induction for $\lt$ ({\sf IND}).  We will first show how to recover full induction.

It is a well known result that, in arithmetical context, (uniformly) reflecting on  $\mathsf{EA}$ suffices to obtain the full induction schema for $\mc L$.  Kreisel and L\'evy, in \cite{krl68},  proved the equivalence of uniform reflection and full induction over {\sf EA} -- that is the equivalence of {\sf EA} plus uniform reflection for {\sf EA} and Peano Arithmetic ({\sf PA}).\footnote{See Beklemishev \cite{bek05}, p.~37 for a proof of this fact.} We will apply Kreisel and Levy's strategy to our setting. In order to do so, however, their original argument has to be modified in several respects. First of all, we allow formulas of $\lt$ and not just of $\mc L$ to appear in instances of the induction schema. In addition, we have to consider an induction {\it rule} because the induction axiom involving the material conditional may fail to be sound in the setting of Basic De Morgan Logic. Finally, already in this step, we shall employ of our generalized reflection rule ${\sf R}_T$ instead of the basic reflection rule ${\sf r}_T$. In what follows we denote as {\sf PA}$_\T$ the version of {\sf PA} formulated in $\lt$ whose logic is {\sf BDM} and in which the truth predicate can appear in instances of induction.

\begin{lemma} $\mathsf{PA_T} \subseteq \mathsf{R ({ \sf Basic})}$ 
\end{lemma}
\begin{proof}
Let $A(x)$ be a formula in $\lt$ with one free variable. We want to show that in $\mathsf{R(\basic)}$ the full induction rule  
	\begin{equation}\label{fuind}
		\AxiomC{$\Gamma, A(x) \Ra A( x+\bar 1 ), \Delta $ }
		\UnaryInfC{$\Gamma, A (0) \Ra A (t), \Delta$}
		\DisplayProof
	\end{equation}
for formulas of $\lt$ is admissible. The following inference is admissible in $\basic$ -- and in fact in predicate logic in $\lt$ only -- for any $n\in \omega$:
	\begin{equation}\label{indeinz}
		\AxiomC{$\Gamma, A(x) \Ra A( x+\bar 1 ), \Delta $ }
		\UnaryInfC{$\Gamma, A (0) \Ra A (\bar n), \Delta$}
		\DisplayProof
	\end{equation}
By \eqref{pr1},  since the proof transformation in \eqref{indeinz} is elementary, $\basic$ proves
	\begin{equation}\label{indzwei}
		\prt_{{\sf Basic}}(\ulc\Gamma, A(x)\Ra A(x+\bar 1),\Delta\urc,\ulc\Gamma, A(0)\Ra A(\dot y),\Delta\urc)
	\end{equation}
Now by assumption, \eqref{indzwei}, and ${\sf R}_{\basic}$ we conclude
	$$
		\Gamma, A(0)\Ra A(y),\Delta
	$$
\end{proof}
The full set of compositional sequents of $ { \sf PKF }$ is obtained by complementing the clauses for the connectives by the ones for quantifiers. This can be achieved by closing the theory ${\sf R}(\tso)$ under ${\sf R}_{{\sf R}(\tso)}$, that is, by performing one iteration of the general reflection rule. 
\begin{lemma}
	${ \sf R }^2( { \sf TS }_0)$ proves $(\T \forall_{1-2})$ and $(\T \exists_{1-2})$.
\end{lemma}
\begin{proof}
We prove \T $\forall_1$; the other cases are treated similarly. For all $\lt$-formulas $A(v)$ with only $v$ free, ${\sf R}(\tso)$ proves 
	\begin{align*}
		&\T[Ay]\Ra A(y) &&\text{by Lemma \ref{uts}}\\
		&\forall y\,\T[Ay]\Ra  \forall y\,A(y) &&\text{by logic}\\
		&\forall y\, \T[Ay]\Ra  \T\ulc\forall y\,A(y)\urc &&\text{by (\T2)}
	\end{align*}
		The argument just carried out in ${\sf R}(\tso)$ can uniformly be formalized in $\basic$, i.e., $\basic$ proves: 
		\[
			{\sf Pr}_{{\sf R}(\tso)}(\ulc\sentt (\subdot\forall y\dot x),\forall y\,\T \dot x(y/v)\Ra \T(\forall y\,\dot x(y/v))\urc)
		\]
	Therefore $\rtwo{\tso}$ suffices to conclude
	\[
		\sentt (\subdot\forall y x),\forall y\,\T  x(y/v)\Ra \T(\forall y\, x(y/v)),
	\]
	as desired.
\end{proof}
\begin{corollary}\label{pkfinc}
$\pkf\subseteq \rtwo{\tso}$. 
\end{corollary}

\noindent
Corollary \ref{pkfinc} shows that two iterations of the generalized rule ${\sf R}_{\tso}$ over our basic theory $\tso$ suffices to recover all compositional truth laws that weren't immediately provable in the original theory as well as the full induction rule for the full language $\lt$. If reflection is considered to be a procedure already implicit in the acceptance of $\tso$, then the laws of $\pkf$ follow naturally from a few applications of this process. However, it is natural to ask whether the inclusion established in Cor. \ref{pkfinc} is proper.

These questions translate, on the conceptual side, into the task of approximating the set of sentences that are valid in the intended models of our theories, which are the Kripke fixed-point models. In doing so, we gather information on how many truth iterations and general claims involving truth we are permitted to assert upon accepting $\tso$ (after reflection)  and how many mathematical patterns of reasoning we regain in the form of transfinite induction. 

\subsection{Recovering transfinite induction by reflection}
\label{subse:trare}
In this section we investigate the question of how much transfinite induction for $\lt$ can be recovered in  iterations of the generalized reflection rule over $\tso$. One of the upshots of our analysis will be that $\rtwo{\tso}$ properly extends ${\sf PKF}$. 

To carry out our proofs, we need to assume a notation system $({\sf OT},\prec)$ for ordinals up to the Feferman-Sch\"utte ordinal $\Gamma_0$ as it can be found, for instance, in \cite{poh09}. ${\sf OT}$ is a primitive recursive set of ordinal codes and $\prec$ a primitive recursive relation on ${\sf OT}$ that is isomorphic to the usual ordering of ordinals up to $\Gamma_0$. We distinguish between fixed ordinal codes, which we denote with $\alpha,\beta,\gamma\ldots$, and $\zeta,\eta,\theta\ldots $ as abbreviations for variables ranging over elements of {\sf OT}. From the results in \cite{hah06} it follows that $\pkf$ proves transfinite induction for $\lt$ only up to any ordinal smaller than $\omega^\omega$. If we focus only on $\mc L$-formulas, however, $\pkf$ proves that much higher ordinals are well-ordered, in particular, $\pkf$ proves the same arithmetical sentences as ${\sf PA}$ plus transfinite induction for $\mc L$ up to any ordinal smaller than $\vphi_{\!\omega}0$. 

Before analyzing how much transfinite induction can be proved in $\rtwo{\tso}$, we introduce some notation. 
 The schema of transfinite induction up to $\alpha$ for the formula $A(v)$ of a language $\mc L_1$ containing $\mc L$ is the rule
\begin{align*}
	&\AxiomC{$\forall \xi\prec \eta\;A(\xi)\Rightarrow A(\eta)$}\RightLabel{${\sf TI}_{\mc L_1}(A,\alpha)$}
	\UnaryInfC{$ \Ra \forall \xi \prec \alpha \;A ( \xi )$}
	\DisplayProof
\end{align*}
\noindent We then denote transfinite induction up to some ordinal $\alpha$ with ${\sf TI}_{\mc L_1} (<\alpha )$, standing for the closure under all rules ${ \sf TI}_{\mc L_1} (A,\beta)$ for $A \in \mc L_1$ and $\beta \prec \alpha$. Analogously, we write ${\sf TI}_{\mc L_1} ( \alpha )$ for the closure under all rules ${ \sf TI}_{\mc L_1} (A,\alpha)$ for $A \in \mc L_1$. In what follows, we will only deal with the cases in which $\mc L_1$ is either $\mc L$ itself or $\lt$. 

As a measure of strength of the theories obtained via iteration of reflection we will mainly focus on how much transfinite induction for $\lt$ is derivable in such theories. However, there is often a direct connection between the amount of  transfinite induction for $\lt$ and $\mc L$ derivable in a truth theory. Both in the case of $\kf$ and $\pkf$, for instance, the amount of transfinite induction for $\lt$ available in the systems -- that is  ${\sf TI}_{\lt} (<\vphi_{1}0 )$ and ${\sf TI}_{\lt} (<\vphi_{0}\omega )$ respectively -- can be used to define classical, Tarskian truth predicates indexed by these ordinals with the crucial contribution of the compositional truth principles of the two theories. This gives a lower bound for the systems in terms of ramified truth hierarchies up to \label{protar}$\vphi_{1}0$ (or $\varepsilon_0$) and $\vphi_{0}\omega$ (or $\omega^\omega$) respectively, which -- by a classical result by Feferman -- yields that $\kf$ and $\pkf$ are proof-theoretically as strong as at least  {\sf PA}$+{\sf TI}_{\lt} (<\vphi_{\varepsilon_0}0 )$ and {\sf PA}$+{\sf TI}_{\lt} (<\vphi_{\omega}0 )$ respectively. 

The following proposition shows that iterating the generalized reflection rule twice over $\tso$ enables us to go beyond $\pkf$. This also gives us more information about the question that was posed on page \pageref{questrul} about the comparison between the rules $({\sf r}_T)$ and $({\sf R}_T)$.  By Proposition \ref{upper}, the theory ${\sf r}(\utso)$ is a subtheory of {\sf PKF}. The next will entail that ${\sf R}^2(\tso)$ is indeed stronger than $\pkf$.

\begin{proposition} \label{pr2ti}
${\sf R}^2(\basic)\vdash {\sf TI}_{\lt} ( \omega^\omega )$
\end{proposition}
\begin{proof}
We first prove in ${\sf R}(\basic)$ that, for all $n\in \omega$, 
\begin{equation}\label{tioo}
\AxiomC{$\Gamma,\forall \zeta\prec \eta\;A(\zeta)\Rightarrow A(\eta),\Delta$}
\UnaryInfC{$\Gamma \Ra \forall \zeta\prec \omega^{n}\;A(\zeta), \Delta$}
\DisplayProof
\end{equation}
To prove \eqref{tioo}, we first prove in ${\sf R}(\basic)$, for all $n\in \omega$:
\begin{equation}\label{preli}
\AxiomC{$\forall \zeta\prec \eta\;A(\zeta)\Rightarrow A(\eta)$}
\UnaryInfC{$\forall \zeta\prec \eta\;A(\zeta)\Rightarrow \forall \zeta\prec \eta+\omega^{n}\;A(\zeta)$}
\DisplayProof
\end{equation} 
We reason as follows in ${\sf R}({\basic})$:
\begin{align}
\label{tiuno}&\forall \zeta\prec \eta\;A(\zeta)\Rightarrow A(\eta)\\
\label{titre}&\forall \zeta\prec \eta\;A(\zeta)\Ra \forall \zeta\prec \eta+\omega^0\;A(\zeta)&&\text{by \eqref{tiuno}}\\
\label{titre}&\forall \zeta\prec \eta\;A(\zeta)\Ra \forall \zeta\prec \eta+\omega^{n}\;A(\zeta)&& \text{external ind. hyp.}\\
\label{titrem}& \forall \zeta\prec \eta+(\omega^n\times x)\,A(\zeta)\Ra \forall \zeta\prec \eta +(\omega^n\times x)+\omega^n \,A(\zeta)&&\text{from \eqref{titre}}\\
\label{tiquattro}&\forall \zeta\prec \eta+(\omega^n\times 0)\;A(\zeta)\Ra \forall x\,\forall \zeta\prec \eta+(\omega^{n}\times x)\;A(\zeta)&&\text{by  ({\sf IND})}\\
\label{tiquattrom}&\forall \zeta\prec \eta\;A(\zeta)\Ra \forall \zeta\prec \eta+\omega^{n+1}\;A(\zeta)
\end{align}
The last two lines give us the induction step and therefore \eqref{preli} by, possibly, a series of cuts.  

Now in $\basic$, 
\begin{equation}
{\sf Pr}^2_{{\sf R}(\basic)}([\forall \zeta\prec \eta\;A(\zeta)\Rightarrow A(\eta)],[\forall \zeta\prec \eta\;A(\zeta)\Rightarrow \forall \zeta\prec \eta+\omega^{x}\;A(\zeta)])
\end{equation} 
Therefore, in ${\sf R}^2(\basic)$,
\begin{equation}\label{tioo1}
\AxiomC{$\forall \zeta\prec \eta\;A(\zeta)\Rightarrow A(\eta)$}
\UnaryInfC{$\forall \zeta\prec \eta\;A(\zeta) \Ra \forall x\,\forall \zeta\prec \eta+\omega^{x}\;A(\zeta)$}
\DisplayProof
\end{equation}
That is 
\begin{equation}\label{tioo1.5}
\AxiomC{$\forall \zeta\prec \eta\;A(\zeta)\Rightarrow A(\eta)$}
\UnaryInfC{$\forall \zeta\prec \eta\;A(\zeta) \Ra \forall \zeta\prec \eta+\omega^{\omega}\;A(\zeta)$}
\DisplayProof
\end{equation} 
But if \eqref{tioo1.5}, by letting $\eta$ to be $0$, we get\begin{equation}\label{tioo2}
\AxiomC{$\Gamma,\forall \zeta\prec \eta\;A(\zeta)\Rightarrow A(\eta),\Delta$}
\UnaryInfC{$ \Gamma\Ra \forall \zeta\prec \omega^{\omega}\;A(\zeta),\Delta$}
\DisplayProof
\end{equation}
\end{proof}
By the proof theoretic analysis of $\pkf$ we know that it can only prove transfinite induction for $\lt$ for ordinals smaller than $\omega^\omega$. But this fact is not dependent in any way on the truth theoretic principles of $\pkf$: already ${\sf PA}_\T$, in fact, proves ${\sf TI}_{\lt}(<\omega^\omega)$. This is also reflected by the fact that Proposition \ref{pr2ti} does not rely on the truth principles of $\tso$. However, by Corollary \ref{pkfinc}, we have:

\begin{corollary}
$\pkf$ is a proper subtheory of ${\sf R}^2(\tso)$.
\end{corollary}
\noindent Transfinite induction up to $\omega^\omega$, however, is clearly not the limit of what we can achieve in ${\sf R}^2(\basic)$.  By using similar methods to the ones employed in Proposition \ref{pr2ti}, and starting from \eqref{tioo1.5}, we can verify that the following rule is admissible in ${\sf R}^2(\basic)$:
	\[
		\AxiomC{$\forall \zeta\prec \eta\;A(\zeta)\Rightarrow A(\eta)$}
		\UnaryInfC{$\forall \zeta\prec \theta\,A(\zeta)\Ra \forall \zeta\prec \theta+\omega^{\omega+k}\, A(\zeta)$}
		\DisplayProof
	\]
\vspace{5pt}
	
\noindent Generalizing this strategy it is possible to show the following:
\begin{lemma}\label{timesn}
 In ${\sf R}^{(n+1)}(\basic)$ the following rule is admissible:
 	\[
		\AxiomC{$\forall \zeta\prec \eta\;A(\zeta)\Rightarrow A(\eta)$}
		\UnaryInfC{$\forall \zeta\prec \theta\,A(\zeta)\Ra \forall \zeta\prec \theta+\omega^{\omega \times n}\, A(\zeta)$}
		\DisplayProof
	\]
\end{lemma}
\begin{proof}
By external induction on $n$.
We have established the claim for $n=1$.
Assume that it holds for $n$.
Then we can argue in ${\sf R}^{(n+1)}(\basic)$:
Assume 
\[ \forall \zeta\prec \eta\;A(\zeta)\Rightarrow A(\eta) \] 
then by the induction hypothesis we have 
\[ \forall \zeta\prec \theta\,A(\zeta)\Ra \forall \zeta\prec \theta+\omega^{\omega \times n}\, A(\zeta) \]
and 
\[ \forall \zeta\prec \theta +\omega^{\omega \times n} \times k \,A(\zeta)\Ra \forall \zeta\prec \theta+\omega^{\omega \times n} \times k +\omega^{\omega \times n}\, A(\zeta). \]
By the induction principle ({\sf IND}), applied on  $k$, we obtain 
\[ \forall \zeta\prec \theta +\omega^{\omega \times n} \times 0 \,A(\zeta)\Ra \forall k\, \forall \zeta\prec \theta+\omega^{\omega \times n} \times k \, A(\zeta) \]
giving us
\[ \forall \zeta\prec \theta \,A(\zeta)\Ra  \forall \zeta\prec \theta+\omega^{\omega \times n} \times \omega \, A(\zeta) \]
which is
\[ \forall \zeta\prec \theta \,A(\zeta)\Ra  \forall \zeta\prec \theta+\omega^{\omega \times (n +1) } \, A(\zeta) \]
\end{proof}

Lemma \ref{timesn} immediately entails that ${\sf R}^{n}(\basic)$ proves ${\sf TI}_{\lt}(<\omega^{\omega\times n})$. Therefore, if we reflect on $\tso$ instead of $\basic$, we are able to define in ${\sf R}^{n}(\tso)$ ramified truth predicates for any ordinal smaller than $\omega^{\omega\times n}$ by following the strategy employed by Halbach and Horsten and described on page \pageref{protar}. 

This strategy can be iterated even further. Ideally, we would like to reach, by as little reflection iterations as possible, the amount of transfinite induction for $\lt$ -- and therefore of ramified truth predicates -- that are available in $\kf$,  the classical counterpart of $\pkf$. However,  we conclude this section by providing only a first, and presumably rather inefficient, approximation to this task.

By letting ${\sf R}^{\omega}(\basic) := \bigcup_{n \in \omega} {\sf R}^{n}(\basic)$, a direct consequence of Lemma \ref{timesn} is that 
\begin{corollary}
In ${\sf R}^{\omega}(\basic)$ we have ${\sf TI}_{\lt} ( < \omega^{(\omega^2 )})$
\end{corollary}
Therefore the theory ${\sf R}^{\omega}(\tso)$ can define ramified truth predicates indexed by all ordinals $\omega^{\omega\times n}$ for all natural numbers $n$. 
 
 Although $\omega$ may seem to be a natural stopping point, the procedure can be iterated even further into the transfinite. Following a well-known tradition initiated by Feferman in \cite{fef62}, the theories ${\sf R}^n(\basic)$ can  all be shown to be recursively enumerable. Moreover, the notion of being a proof in ${\sf R}^n(\basic)$ is recursive. We can then find a primitive recursive function enumerating all those proof predicates. By employing the recursion theorem, therefore, we can find an index for this enumeration that can be used to formalize, via a recursive predicate, the notion of being a proof employing rules proper of one of the theories ${\sf R}^n(\basic)$. This, however, suffices to formulate the notion of being a proof in ${\sf R}^{\omega}(\basic)$: clearly,  similar procedure can be extended at least to ordinals smaller than $\varepsilon_0$.
 
But once a recursive formalization of transfinite iterations of our reflection rules is available, it becomes clear that enough iterations of reflection over $\basic$ will lead us to the amount of transfinite induction for $\lt$ available in $\kf$. By letting $\omega_0:=1$, and $\omega_{n+1}:=\omega^{\omega_n}$, we have, rather unsurprisingly,  
	\begin{observation}
		${\sf R}^{\omega_n+1}(\basic)\vdash {\sf TI}_{\mc L_\T}(\omega_{n})$
	\end{observation}


\section{Conclusion}

\noindent Starting with principles that are minimally constitutive of the notion of truth, such as the initial sequents of the theory $\tso$, we have investigated the result of iterating reflection rules  over them. A similar project, in the context of classical logic and therefore without the basic principles of $\tso$, has been recently pursued by Horsten and Leigh \cite{hol15}. We claim that for two reasons our nonclassical setting provides a more coherent framework for such a project for two main reasons. First, in a \emph{classical} setting the interderivability of $A$ and $\T\ulc A\urc$ (which is the defining characteristic of $\tso$) cannot be consistently maintained. Second, following a theme by Kreisel, the global reflection ${\sf GRF}_T$ for a theory $T$ is the \emph{intended} soundness extension of $T$. Other  proof theoretic reflection principles, including the uniform reflection principle ${\sf RFN}_T$, are only justified by an appeal to global reflection. However, as shown in \S2.3, in classical axiomatizations of Kripke's fixed point constructions, the use of the global reflection principle is at odds with the overall strategy of iterating reflection rules. 

One way to understand the results of this article is by asking which statements $\tso$ and the result of iterating reflection rules over it can prove to be true, i.e., by considering their provable sequents of the form $\Ra \T\ulc A\urc$ for $A$ in $\lt$ or, in short, at their {\it truth theorems}. $\bdm$ in itself has no theorems at all. When initial sequents for identity are added to it as well as arithmetical initial sequents, even if the truth predicate is in the signature of the theory, one only obtains arithmetical theorems but no truth theorems. $\tso$, by contrast, does prove truth theorems, but only truth theorems of the form 
\[
	\underbrace{\T\ldots\T}_{\text{n-times}}\ulc A\urc
\]
where $A$ is an arithmetical theorem of $\basic$. This shortcoming of $\tso$ is accompanied by the lack of other desirable properties of the theory, such as full compositionality (see again \S2.3). By adding a uniform or global \emph{reflection rule} to $\tso$, we restore our full capability of reasoning inductively with the truth predicate, and several compositional truth sequents. Full compositionality, together with the possibility of establishing theorems of the form
\[
	\underbrace{\T\ldots\T}_{\text{$\omega^{\omega+n}$-times}}\ulc A\urc
\]
for $A$ again an arithmetical theorem of $\basic$, is reached when we consider the theory ${\sf R}^2(\tso)$, i.e., via a further iteration of the generalized reflection rule ${\sf R}_T$ over $\tso$. At this stage, we already recapture and surpass all  truth theorems of the full compositional theory $\pkf$. A natural goal for the process of iteration may be to reach the truth theorems of the classical theory $\kf$ (or equivalently, $\pkf+{\sf TI}_{\lt}(<\varepsilon_0)$, as shown in \cite{nic17}). This can be achieved via suitable transfinite progressions of theories obtained by reflection over $\tso$. 

From a semantic perspective, there is a tight match between the truth theorems of our theories and the levels of the construction of the minimal fixed point of Kripke's construction from \cite{kri75}. By extending $\tso$ with an $\omega$-rule, this connection can be made explicit: the theorems of $\tso$ plus the $\omega$-rule are exactly the $\lt$-sentences that are in the extension of the truth predicate in the minimal fixed point of Kripke's theory (see \cite{fig16} for a recent proof). Uniform reflection principles are recursive approximations of the $\omega$-rule. Therefore iterations of reflection, and the corresponding truth theorems of the resulting theories, can be seen as approximations to the full $\omega$-rule added to $\tso$ as they represent initial stages of the construction of the minimal fixed point. It is also clear that all the theories that we have considered are {\it internal} axiomatizations of Kripke fixed points. Therefore the hierarchy that we have studied can also be seen as an attempt to capture, via recursively axiomatized theories, the set of grounded sentences first isolated by Kripke.

Nonetheless our work leaves many open questions and possibilities for improvement: from a technical point of view,  a sharper proof-theoretic analysis of the theories obtained by iterated reflection would be desirable to see clearly, for instance, how much one can obtain with finite iterations of reflection. Moreover, it would be interesting to see whether the reflection rules can be strengthened via `higher-order' reflection rules in such a way that only finitely many iterations of them could suffice to reach the truth theorems of $\kf$. Finally, there remains the question whether the gap between ${\sf TI}_{\lt}(<\omega^\omega)$ and ${\sf TI}_{\lt}(<\varepsilon_0)$ -- which is determined by whether {\sf PA} in the signature of $\lt$ is formulated in {\sf BDM} or classical logic respectively -- can be closed by supplementing $\bdm$ with a suitable conditional in such a way that the conceptual advantages of the treatment of truth in $\tso$ are preserved. 

 \bibliographystyle{abbrv}
\bibliography{litHB}

\end{document}